\documentclass[sn-mathphys,Numbered]{sn-jnl}

\usepackage{graphicx}%
\usepackage{multirow}%
\usepackage{amsmath,amssymb,amsfonts}%
\usepackage{amsthm}%
\usepackage{mathrsfs}%
\usepackage[title]{appendix}%
\usepackage{xcolor}%
\usepackage{textcomp}%
\usepackage{manyfoot}%
\usepackage{booktabs}%
\usepackage{algorithm}%
\usepackage{algorithmicx}%
\usepackage{algpseudocode}%
\usepackage{listings}%

\theoremstyle{thmstyleone}%
\newtheorem{theorem}{Theorem}
\newtheorem{lemma}[theorem]{Lemma}
\theoremstyle{thmstyletwo}%

\theoremstyle{thmstylethree}%
\newtheorem{definition}{Definition}%

\begin{document}

\title[Article Title]{Optimal Hypothesis Testing Based on Information Theory}


\author*[1]{\fnm{Dazhuan} \sur{Xu}}\email{xudazhuan@nuaa.edu.cn}

\author[2]{\fnm{Nan} \sur{Wang}}\email{n\_wang@nuaa.edu.cn}

\affil[1,2]{\orgdiv{College of Electronic and Information Engineering}, \orgname{Nanjing University of Aeronautics and Astronautics}, \orgaddress{\street{29 General Avenue, Jiangning District}, \city{Nanjing}, \postcode{210000}, \country{China}}}


\abstract{There has a major problem in the current theory of hypothesis testing in which no unified indicator to evaluate the goodness of various test methods since the cost function or utility function usually relies on the specific application scenario, resulting in no optimal hypothesis testing method. In this paper, the problem of optimal hypothesis testing is investigated based on information theory. We propose an information-theoretic framework of hypothesis testing consisting of five parts: \textit{test information} (TI) is proposed to evaluate the hypothesis testing, which depends on the \textit{a posteriori} probability distribution function of hypotheses and independent of specific test methods; \textit{accuracy} with the unit of $ bit $ is proposed to evaluate the degree of validity of specific test methods; the \textit{sampling \textit{a posteriori}} (SAP) probability test method is presented, which makes stochastic selections on the hypotheses according to the \textit{a posteriori} probability distribution of the hypotheses; the \textit{probability of test failure} is defined to reflect the probability of the failed decision is made; \textit{ test theorem} is proved that all accuracy lower than the TI is achievable.  Specifically, for every accuracy lower than TI, there exists a test method with the probability of test failure tending to zero. Conversely, there is no test method whose accuracy is more than TI. Numerical simulations are performed to demonstrate that the SAP test is asymptotically optimal. In addition, the results show that the accuracy of the SAP test and the existing test methods, such as the maximum \textit{a posteriori} probability,  expected \textit{a posteriori} probability, and median \textit{a posteriori} probability tests, are not more than TI.}




\maketitle

\section{Problem of Optimal hypothesis testing}\label{sec1}

Bayesian statistical inference \cite{r1,r2} is an important technique in mathematical statistics in which Bayes? theorem\cite{r3} is used to update the probability for a hypothesis as more evidence or information becomes available. Bayesian methods allow the incorporation of \textit{a priori} knowledge into statistical inference, which is recognized as critical in practical applications, such as disease diagnosis and drug testing. In statistical inference, hypothesis testing is a major class of problems, and is usually handled using maximum a posteriori tests (MAP),  expected \textit{a posteriori} (EAP), and median \textit{a posteriori}  (MeAP)  test methods under the Bayesian framework. To introduce what we will discuss next, a classic example of a coin flip is given.

\subsection{Example: hypothesis testing problem in binomial distribution}

\subsubsection{Formulation problem}

In the case of a coin toss, assume that the probability of heads up is $ \theta $. If we toss the coin $ n \in \mathbb{N} $  ($ \mathbb{N}=\left\{ {1,2, \cdots ,N} \right\}   $)  times and $ k $  times it comes up heads. The probability distribution function(PDF) of $ k $ is given by
\begin{equation}
	P\left( {k} \right) = C_n^k{\theta ^k}{\left( {1 - \theta } \right)^{n - k}}.\label{1}
\end{equation} 
Now,  the coin toss event is transformed into a hypothesis testing problem for the test of $ n $. That is, make statistical inferences regarding hypotheses $ H_{0}: n=1 $, $ H_{1}: n=2 $,  \ldots, $ H_{\rm{N}}: n=N $. Then, Eq. (\ref{1}) is rewritten as 
\begin{equation}
	P\left( {k|n} \right) = C_n^k{\theta ^k}{\left( {1 - \theta } \right)^{n - k}}.\label{2}
\end{equation} 
In the framework of Bayesian statistics, we first give the $ a priori $ distribution of $ n $. Suppose $ n $  obeys the distribution 
\begin{equation}
	\pi \left( n \right) = \frac{1}{N}.\label{3}
\end{equation}
Then, by Bayesian formulation, the \textit{a posteriori} PDF is given by
\begin{equation}
	P\left( {n|k} \right) = \frac{{\pi \left( n \right)C_n^k{\theta ^k}{{\left( {1 - \theta } \right)}^{n - k}}}}{{\sum\limits_{n = 1}^M {\pi \left( n \right)C_n^k{\theta ^k}{{\left( {1 - \theta } \right)}^{n - k}}} }} = \frac{{C_n^k{\theta ^k}{{\left( {1 - \theta } \right)}^{n - k}}}}{{\sum\limits_{n = 1}^M {C_n^k{\theta ^k}{{\left( {1 - \theta } \right)}^{n - k}}} }}.\label{4}
\end{equation}

\subsubsection{Existing test methods}
The MAP test and the EAP test are commonly used hypothesis testing methods. The MAP test is to select the hypothesis with the highest \textit{a posteriori} probability, i.e.,
\begin{equation}
	{\hat n_{{\rm{MAP}}}} = \arg \mathop {{\rm{max}}}\limits_n P(n|k).\label{5}
\end{equation}
The EAP test is to select the hypothesis that is closest to the expectation of the \textit{a posteriori} probability, i.e., 
\begin{equation}
	{\hat n_{{\rm{EAP}}}} = \arg \mathop {{\rm{min}}}\limits_n \left| n-E\left[ P(n|k) \right]  \right| .\label{6}
\end{equation}
The MeAP test is used to find the hypothesis that makes the accumulated value of the \textit{a posteriori} probabilities is closest to $ \dfrac{1}{2} $, i.e.,
\begin{equation}
	{{\hat n}_{{\rm{MeAP}}}} = \arg \mathop {\min }\limits_n \left| {\sum\limits_{n = 1}^N {P\left( {n\left| k \right.} \right)}  - \frac{1}{2}} \right|.
\end{equation}

The results of the MAP, EAP, and MeAP tests for given $ k $ are provided in Fig. 1.
\begin{figure}
	\includegraphics[width=5in]{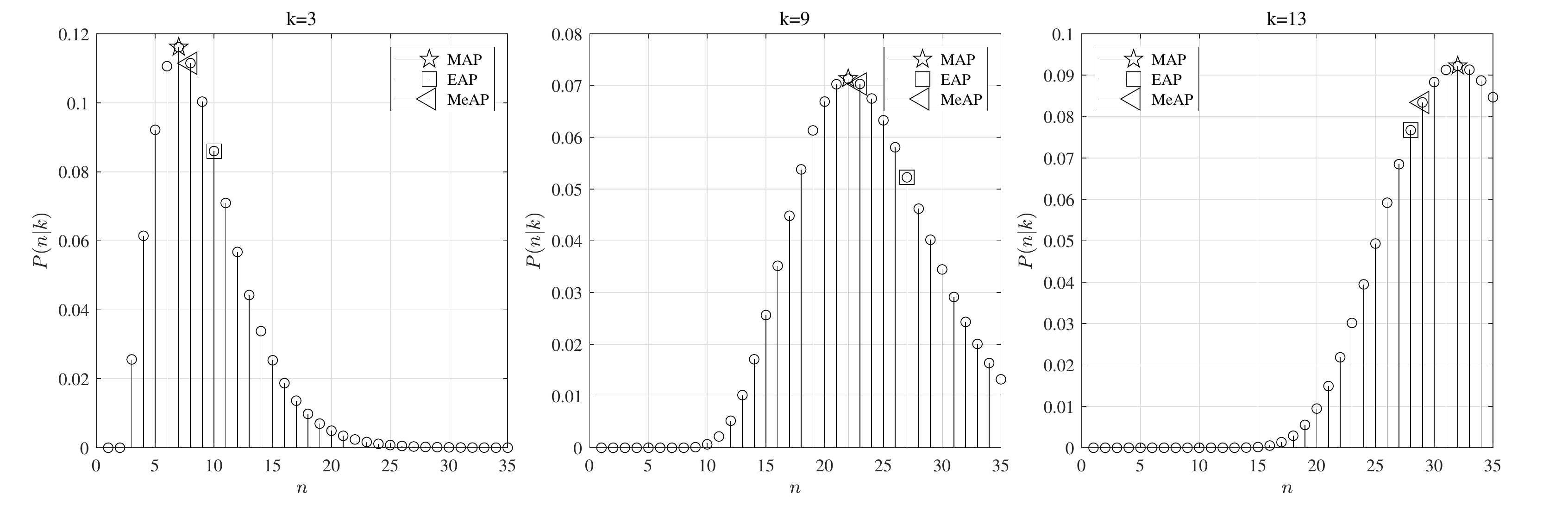}
	\caption{The \textit{ a posteriori} PDF of $ n $ and the results of MAP, EAP, and MeAP tests for given $ N=35 $, $ \theta=0.4 $ and $ k=3,9,13 $.}
	\label{penG}
\end{figure}

It can be seen that for $ k $ equal to $ 3 $, the results of MAP, EAP, and MeAP tests are hypotheses $ n=7 $, $ n=10$, and $ n=8 $, respectively. Similarly, for $ k=9 $, the MAP, EAP, and MeAP tests correspond to the hypothesis $ n=22 $, $ n=27 $, and $ n=23 $, respectively; for $ k=13 $, the MAP, EAP, and MeAP tests correspond to the hypothesis $ n=32 $, $ n=28 $, and $ n=29 $, respectively. 

\subsubsection{The evaluation of the results}
Bayesian hypothesis testing is usually evaluated by the probability of error, 
\begin{equation}
	{P_{e}} = \rm{Pr} ( \hat n \ne n ). 
\end{equation}
According to the rules of the MAP, EAP, and MeAP tests, the average probability of error of the MAP is smaller than that of the the EAP and MeAP tests. 

In addition to the error probability, a cost function or utility function can be established according to the specific application scenario to find the hypothesis that minimizes the loss or maximizes the utility.

\subsection{Problem and Challenges}
From the example, we find the following problems to be solved in hypothesis testing.  Firstly, there is no unified indicator to evaluate the goodness of various test methods since the cost function (probability of error) or utility function usually depends on the specific application scenario, resulting in no optimal hypothesis testing method.  Given the issue, hypothesis testing faces the following challenges: whether there is a general indicator independent of the specific applications; whether there is an optimal test method; and whether the general indicator is achievable.

\subsection{Contributions}
In this paper, the information theoretical framework for hypothesis testing is presented based on information theory \cite{r4}. Test information (TI) is defined as the difference between the \textit{a priori} and the \textit{a posteriori} entropies of hypotheses, which depends on the \textit{a posteriori} probability of hypotheses and is independent of specific test methods. We propose the  accuracy to evaluate specific test methods. We propose the sampling \textit{a posteriori} (SAP) test, which makes selections on the hypotheses according to the \textit{a posteriori} probability distribution of the hypotheses.  The probability of test failure is defined to reflect the probability of the failed decision is made. The test theorem is proved based on the SAP test and the Asymptotic Equipartition Property (AEP), stating that all accuracy lower than the TI are achievable.  Specifically, for every accuracy lower than TI, there exists a test method with the probability of test failure tending to zero. Conversely,  there is no test method whose accuracy is more than TI. The proof is inspired by the coding theorem and TI is analogous to Shannon's capacity. Numerical simulations are performed to demonstrate that the accuracy of the SAP test approaches the TI, hence the SAP test is asymptotically optimal. In addition, we compare the TI to the accuracy of the MAP,  EAP, MeAP, and SAP tests. The result shows that the accuracy of the these test methods are not more than TI.

\subsection{Organization}
The rest of the paper is organized as follows. In Section 2, the information theoretic framework of  hypothesis testing is proposed. The simulation is provide in Section 3, and Section 4 concludes this paper.

\section{Information Theoretic framework of  hypothesis testing}\label{sec2}
The uniqueness of information theory is that it can measure the amount of information a variable receives from another variable. Hypothesis testing is essentially a process of measuring uncertainty reduction of a hypothesis. Therefore, it is reasonable to apply information theory to hypothesis testing. This section provides an information-theoretic framework for hypothesis testing.

\subsection{Test Information}
A general hypothesis testing system model $\left( {\mathbb{X}, \pi (x), p(y|x), {\bf\mathbb{Y}} , \hat x = t(\cdot)} \right)$ is shown in Fig. 2.
\begin{figure*}[!ht]
	\centering
	\includegraphics[width=2.7in]{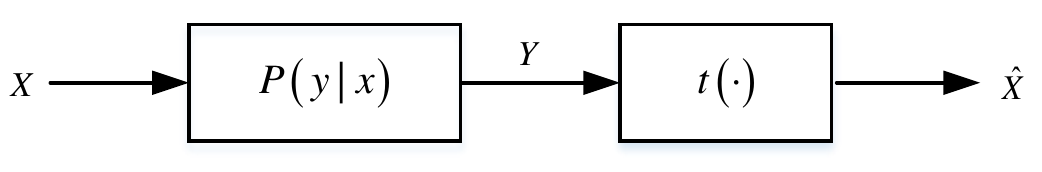}
	\begin{center}
		\caption{Hypothesis test model}
	\end{center}
	\label{f1}
\end{figure*}
$\mathbb{X} \left(  X \!\in\! \bf\mathbb{X}\right)  $ is an input set;  $ {\bf\mathbb{Y}}\left(  Y \!\in\! \bf\mathbb{Y}\right)  $ is a vector space in the complex regime; $\hat X = t(\cdot)$ is  a test that is a function of the observed data $Y $, and $ P(y|x) $ is the conditional PDF.  The \textit{a posteriori} PDF $  P\left( x|y \right) $ is used to test hypotheses since all these statistical properties are contained in the  $  P\left( x|y \right) $. Based on the model of hypothesis testing, the definition of TI is given.
\begin{definition}[Test Information] \label{2.1}
	The mutual information $I\left( {Y;X} \right) $ between the state variable $ x $ and the $ y $ is called the TI,
	\begin{equation}
		I\left( {Y;X} \right) = {\rm{E}}\left[ {\log \frac{{P\left( {\left. {y} \right|x} \right)}}{{P\left( {y} \right)}}} \right].
	\end{equation}	
\end{definition}

The definition of the TI is rewritten as 
\begin{equation}
	I\left( {Y;X} \right) = H\left( X \right) - H\left( {X|Y} \right),
\end{equation}
where 
\begin{equation}
	H\left( X \right) =  - \sum\limits_{i = 1}^m {\pi ({x_i})\log \pi ({x_i})} 
\end{equation}
is the \textit{a priori}  entropy and
\begin{equation}
	H\left( {X|Y} \right) =  - \sum\limits_{i = 1}^m {P({x_i}|{y_i})\log P({x_i}|{y_i})} 
\end{equation}
is the \textit{a posteriori} entropy.

TI is a theoretical indicator for quantifying the goodness of the test's results, which is independent of any test method. The more TI there is, the better the test result will be.

\subsection{Sampling A Posteriori Probability Test}

In addition to the three hypothesis testing methods including the MAP, EAP, and MeAP tests, the test problem can be also addressed by the random sampling method.  In general, sampling refers to the random selection of $ n $ independent and identically distributed samples from the total population. Consider that the properties of the hypothesis are embodied by the \textit{a posteriori} PDF $  P\left( {x\left| {y} \right.} \right) $ based on the Bayes formula, we propose the SAP test based on $  P\left( {x\left| {y} \right.} \right) $. Specifically, the test results of the SAP test satisfy
\begin{equation}
	{\hat x_{{\rm{SAP}}}} = \arg \mathop {\rm{smp}}\limits_x P\left( {x\left| {y} \right.} \right),
\end{equation}	where $  \mathop {\rm{smp}}\limits_x \left\{  \cdot  \right\} $ denotes the sampling operator, which selects the hypotheses according to the \textit{a posteriori} probabilities. To understand the SAP test more intuitively,  the sampling results of the MAP, EAP, MeAP, and SAP tests in the introduction are shown in Fig. 3.

\begin{figure}
	\includegraphics[width=5in]{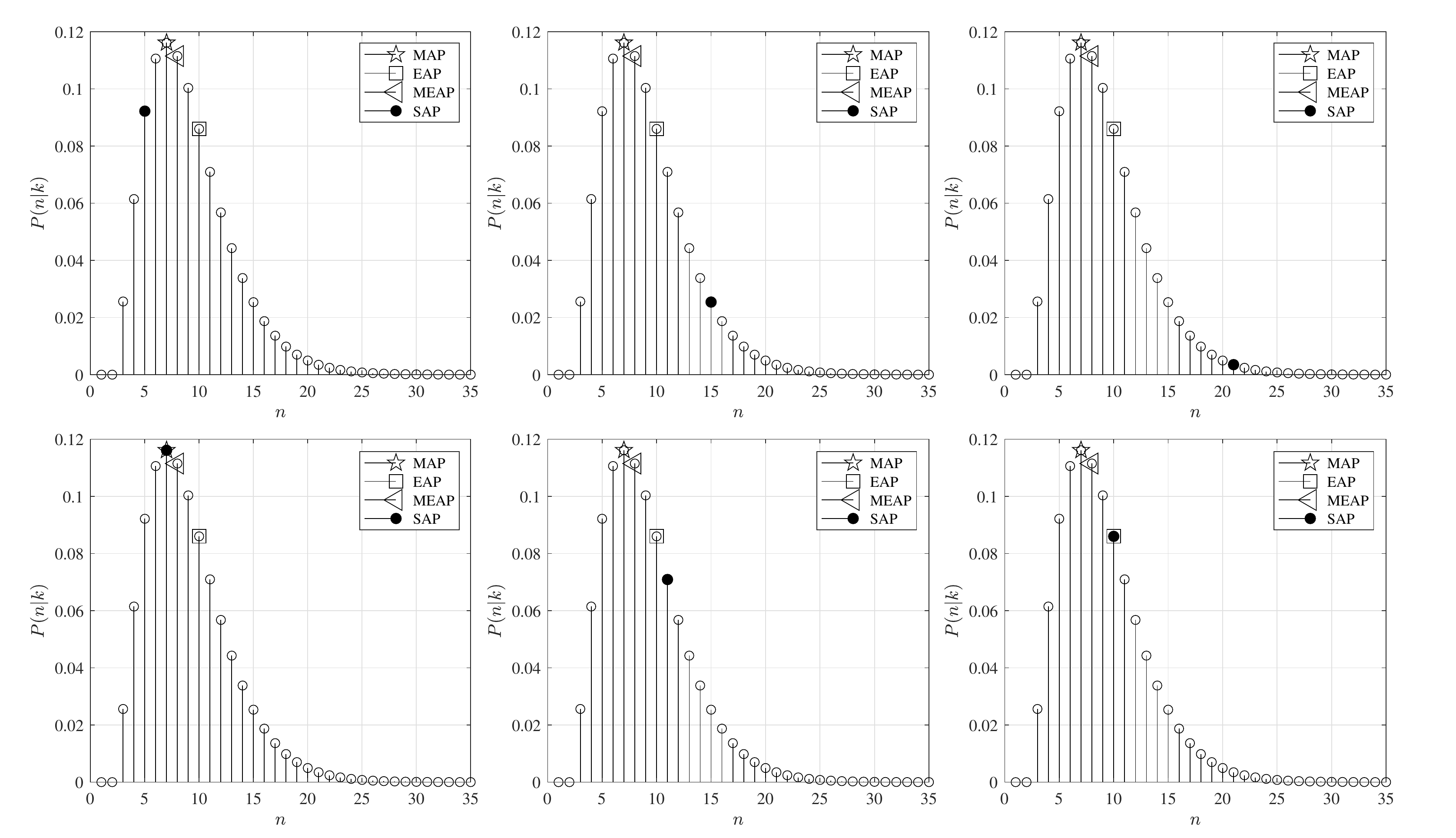}
	\caption{The \textit{ a posteriori} PDF of $ n $ and the results of MAP, EAP, MeAP, and SAP tests for a given $ k $ ($ \theta =0.4 $) .}
	\label{penG}
\end{figure}

As can be seen, for the given $ k =3$,  the results of the MAP, EAP,  and MeAP tests are fixed, i.e., $ n=7 $, $ n=10 $, and $ n=8 $.  On the contrary, the SAP test selects the hypothesis based on the overall probability distribution $ P\left( {n\left| {k} \right.} \right)  $. As a result, the hypotheses can be selected stochastically according to its \textit{a posteriori} probability.

\subsection{Accuracy}


The accuracy of a specific test depends on its empirical entropy, i.e., the empirical \textit{a posteriori} PDF. According to the introduction of the SAP test in the previous section, it is known that its empirical \textit{a posteriori} PDF is the theoretical \textit{a posteriori} PDF. Whereas the empirical \textit{a posteriori} PDFs of MAP, EAP, and MeAP test methods can be obtained statistically by corresponding specific rules, respectively. The specific definition of accuracy we will give in the next subsection.

\subsection{Test Theorem}
In this section, we prove the test theorem from the achievability and the converse results. The proof of the theorem is inspired by Shannon's coding theorem. However, hypothesis testing differs from the coding theorem in that all possibilities for the true hypotheses (corresponding to the code book in coding theorem) are not known. Therefore, the proof given below uses the properties of the typical set, whose basic idea is the random test. 

\begin{figure*}[!ht]
	\centering
	\includegraphics[width=2.7in]{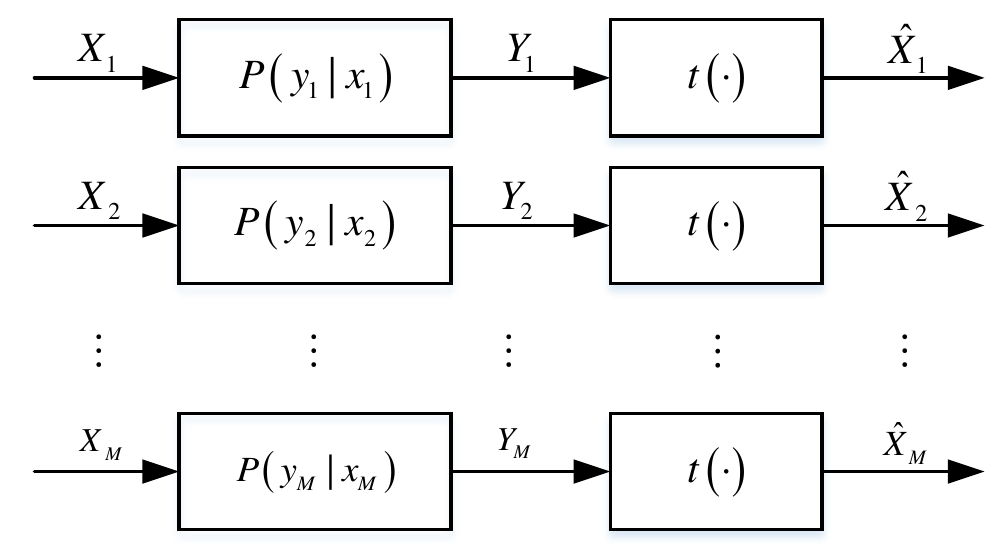}
	\begin{center}
		\caption{Hypothesis testing model for $ M$-th extension}
	\end{center}
	\label{f1}
\end{figure*}
Consider a $ M $ extension of hypothesis test system is denoted as $\left( {\mathbb{X}^{M},\pi (x^{M}),P({y}^{M}|x^{M}),\mathbb{Y}^{M}} \right) $ as shown in Fig. 4, where $ \pi (x^{M}) $ denotes the \textit{a priori} distribution of the hypothesis $ X^{M} $, $ P({y}^{M}|x^{M}) $ represents the conditional PDF, $ \mathbb{X}^{M} $ and $ \mathbb{Y}^{M} $ are data space of the hypothesis $ X^{M}$  and the observed data $ Y^{M}$, respectively. The observed data is gained through $ \pi(x) $ and  $ P({y}|x)  $. Then the \textit{a posteriori} PDF $ P\left( {x\left| {y} \right.} \right)  $ is obtained and a test $ t(\cdot) $ is performed to make a decision $ {\hat{X}} $. It can be seen that $\left( {{X^M},Y^{M},{{\hat X}^M}} \right) $  forms a Markov chain. $  X^M $ and  $  P({{y}^M}|{x^M}) $ satisfy 

\begin{equation}
	\pi \left( {{x^M}} \right) = \prod\limits_{m= 1}^M {\pi \left( {{x_{m}}} \right)} ,
\end{equation}
\begin{equation}
	P({{y}^M}|{x^M}) = \prod\limits_{m= 1}^M {P({{y}_{m}}|{x_{m}})} . 
\end{equation}

Before proving the theorem, we introduce the definitions and lemmas required for developing the test theorem.

\begin{lemma}[Chebyshev Law of Large Numbers]\label{2.2}
	If the random sequence ${Z_1},{Z_2}, \cdots {Z_M} $ are independent identically distributed (i.i.d.) with mean $ \mu  $ and variance ${\sigma ^2} $, where the sample mean is ${\bar Z_M} =\displaystyle \frac{1}{M}\sum\nolimits_{m = 1}^M {{Z_m}} $, then
	\begin{equation}
		\Pr \left\{ {\left| {{{\bar Z}_M} - \mu } \right| > \varepsilon } \right\} \le \frac{{{\sigma ^2}}}{{M{\varepsilon ^2}}}.
	\end{equation}	
\end{lemma}

\begin{lemma}[AEP] \label{2.3}
	If the random sequence ${X_1},{X_2}, \cdots {X_M} $ are i.i.d. $~ \pi\left(x \right)  $, then
	\begin{equation}
		- \frac{1}{M}\log \pi ({X_1},{X_2}, \cdots ,{X_M}) \to {\rm{E}} \left[ { - \log \pi (x)} \right] = H(X) \quad  in \ probability. 
	\end{equation}
\end{lemma}

\begin{definition}[Typical Set] \label{2.4}
	If the random sequence ${X_1},{X_2}, \cdots {X_M} $ are i.i.d. ${\sim} \pi \left( x \right) $, for any $\varepsilon  > 0$, the typical set is defined as
	\begin{equation}
		\mathbb{A} _\varepsilon ^{(M)}(X) = \left\{ {\left( {{X_1},{X_2}, \cdots ,{X_M}} \right) \in {X^M}:\left| { - \frac{1}{M}\log \pi \left( {{X_1},{X_2}, \cdots ,{X_M}} \right) - H(X)} \right| < \varepsilon } \right\}.
	\end{equation}where $\pi \left( {{X_1},{X_2}, \cdots ,{X_M}} \right) = \prod\limits_{m = 1}^M {\pi ({x_m})}  $
\end{definition}
As a result of the AEP, the typical set  $ \mathbb{A} _\varepsilon ^{(M)}(X) $ has the following properties:
\begin{lemma}\label{2.5}
	For any $\varepsilon  > 0$, when $ M $ is sufficiently large, there is
	
	$ \left( 1\right)  $ 
	$ \Pr \left\{ {X \in \mathbb{A}_\varepsilon ^{(M)}(X)} \right\} > 1 - \varepsilon  $
	
	$ \left( 2\right)  $ 
	${2^{ - M\left[ {H(X) + \varepsilon } \right]}} < \pi (x) < {2^{ - M\left[ {H(X) - \varepsilon } \right]}} $
	
	$ \left( 3\right)  $ 
	$(1 - \varepsilon ){{\rm{2}}^{M(H(X) + \varepsilon )}}{\rm{ < }}\left\| {\mathbb{A} _\varepsilon ^{(M)}(X)} \right\| < {{\rm{2}}^{M(H(X) + \varepsilon )}}$, where $ \left\| {\mathbb{A} _\varepsilon ^{(M)}(X)} \right\|  $ denotes the cardinal number of the set $\mathbb{A} _\varepsilon ^{(M)}(X) $, i.e., the number of elements in the set $\mathbb{A} _\varepsilon ^{(M)} (X)$.
\end{lemma}
\noindent\textit{Proof: see (\cite{r4}, pp. 51-53).}

\begin{definition}[Jointly Typical Sequences]\label{2.6}
	The set $ \mathbb{A}_\varepsilon ^{\left( M \right)}(X,Y)$ of jointly typical sequences $\left( {{X^M},{Y^M}} \right)$ with respect to the distribution $P(x,y)$ is the set of $ M $-sequences with an empirical entropies $ \varepsilon $ close to the true entropies, i.e., 
	\begin{equation}
		\begin{split}
			\begin{array}{l}
				\mathbb{A}_\varepsilon ^{\left( M \right)} (X,Y)= \Bigg\{ \left( {{x^M},{{y}^M}} \right) \in {\mathbb{X}^M}\! \times {\mathbb{Y}^M}:\\
				\left| { - \displaystyle\frac{1}{M}\sum\limits_{m = 1}^M {\log P\left( {{x_{m}}} \right)}  - H\left( X \right)} \right| < \varepsilon \\
				\left| { - \displaystyle \frac{1}{M}\sum\limits_{m = 1}^M {\log P\left( {{{y}_{m}}} \right)}  - H\left( Y \right)} \right| < \varepsilon \\
				\left| { - \displaystyle \frac{1}{M}\sum\limits_{m = 1}^M {\log P\left( {{x_{m}},{{y}_{m}}} \right)}  - H\left( {X,Y} \right)} \right| < \varepsilon \Bigg\} ,
			\end{array}
		\end{split}
	\end{equation}
\end{definition} 
\noindent where
\begin{equation}
	P({{y}^M},{x^M}) = \prod\limits_{m = 1}^M {P({{y}_{m}},{x_{m}})} . 
\end{equation}

\begin{lemma}\label{2.7}	
	The typical set is used to test  $ { X^M} $. The extended \textit{a posteriori} PDF of the SAP test is $ {P_{{{ \rm SAP}}}}\left( {{{\hat x}^M}|{{y}^M}} \right) = P\left( {{{\hat x}^M}|{{y}^M}} \right)$  because of the extensions independent of each other. Then the joint \textit{a posteriori} PDF of the SAP test satisfies 
	\begin{equation}
		\begin{array}{c}
			{P_{{\rm{SAP}}}}\left( {{{\hat x}^M},{{y}^M}} \right) = P\left( {{{y}^M}} \right){P_{\rm SAP}}\left( {{{\hat x}^M}|{{y}^M}} \right)
			=P\left( {{{y}^M}} \right)P\left( {{{\hat x}^M}|{{y}^M}} \right)
			= P\left( {{{\hat x}^M},{{y}^M}} \right).
		\end{array} 
	\end{equation}
\end{lemma}

\begin{lemma}[Joint AEP]\label{2.8}
	Let the sequence $ ({X^M},{Y^M}) $ are i.i.d.  $ {\sim}  P({x^M},{y}^M) = \prod\limits_{m = 1}^M {P({x_m},{y}_m)}  $, for any $\varepsilon  > 0$, if $ M $  is sufficiently large, then
	
	$ \left( 1\right)  $ 
	$\Pr \left\{ {({X^M},{Y^M}) \in \mathbb{A}_\varepsilon ^{\left( M \right)}\left( {X,Y} \right)} \right\} \ge 1 - \varepsilon  $
	
	$ \left( 2\right)  $ 
	$ {2^{ - M\left[ {H(X,Y) + \varepsilon } \right]}} < P(x,{y}) < {2^{ - M\left[ {H(X,Y) - \varepsilon } \right]}}  $
\end{lemma}
\noindent\textit{Proof: see (\cite{r4}, pp. 195-197).}

\begin{lemma}[Conditional AEP]	\label{2.9}
	For any $\varepsilon  > 0$, if $ M $  is sufficiently large, then
	
	$ \left( 1\right)  $ 
	${2^{ - M\left( {H(Y|X) + 2\varepsilon } \right)}} < P({y}|x) < {2^{ - M\left( {H(Y|X) - 2\varepsilon } \right)}},$
	
	\indent\indent ${2^{ - M\left( {H(X|Y) + 2\varepsilon } \right)}} < P(x|{y}) < {2^{ - M\left( {H(X|Y) - 2\varepsilon } \right)}},  $ where $(X,Y) \in 	\mathbb{A}_\varepsilon ^{(M)}(X,Y) $.
	
	$ \left( 2\right)  $ Let $ 	\mathbb{A}_\varepsilon ^{(M)}(X|Y) = \left\{ {X:(X,Y) \in 	\mathbb{A}_\varepsilon ^{(M)}(X,Y)} \right\} $ be a set of all $ X^{M} $ that form jointly typical sequences with $ Y^{M} $, then 
	\begin{equation}
		\left( {1 - \varepsilon } \right){2^{M\left( {H(X|Y) - 2\varepsilon } \right)}} < \left\| {	\mathbb{A}_\varepsilon ^{(M)}(X|Y)} \right\| < {2^{M\left( {H(X|Y) + 2\varepsilon } \right)}}.
	\end{equation}
	
	$ \left( 3\right)  $ Let $ X^{M} $ be the jointly typical sequences with $ Y^{M} $, then
	\begin{equation}
		\Pr \left\lbrace  {{X^M} \in 	\mathbb{A}_\varepsilon ^{(M)}(X|Y)} \right\rbrace  > 1 - 2\varepsilon.
	\end{equation}
\end{lemma}
\noindent\textit{Proof: see (\cite{r4}, pp. 386-387).}

The typical sets and the conditional typical set given the observed sequence $ Y^{M} $ are as shown in Fig. 5. $ \left\| {{X^M}} \right\| $ and $ \left\| {{Y^M}} \right\| $ essentially denote the infinite sets contain all of the possible sequences of  $ X^{M} $  and $ Y^{M} $, respectively. Here, the two are considered finite sets. $\mathbb{A}_\varepsilon ^{(M)}\left( X \right) $ and $\mathbb{A}_\varepsilon ^{(M)}\left( {Y} \right) $ represent the typical sets contain all typical sequences of  $ X^{M} $  and $ Y^{M} $, respectively. The sequence  $ X^{M} $ is passed through $ p\left( {y}^{M}|x^{M}\right)  $ to the observer, who then make decisions $ {\hat{X} }^{M}$ based on the typical set test. Specifically, the typical set test finds  the $ {\hat{X} }^{M}$ that  forms the jointly typical sequence with $ {Y}^{M} $ given $ {Y}^{M} $. Since $ {\hat{X} }^{M}$ is not unique as illustrated in Fig. 5, the cardinal number of the typical set  $\mathbb{A}_\varepsilon ^{(M)}\left( { X|Y} \right) $ reflects the performance of the hypothesis testing. Therefore, unlike in Shannon's coding theorem, we cannot evaluate test results in terms of correct or not. Rather, successful and failed are used to assess the test results, and the test performance is measured only when the result is successful. Below, we give the two concepts of successful and failed test results.

\begin{figure*}
	\centering
	\includegraphics[width=3.5in]{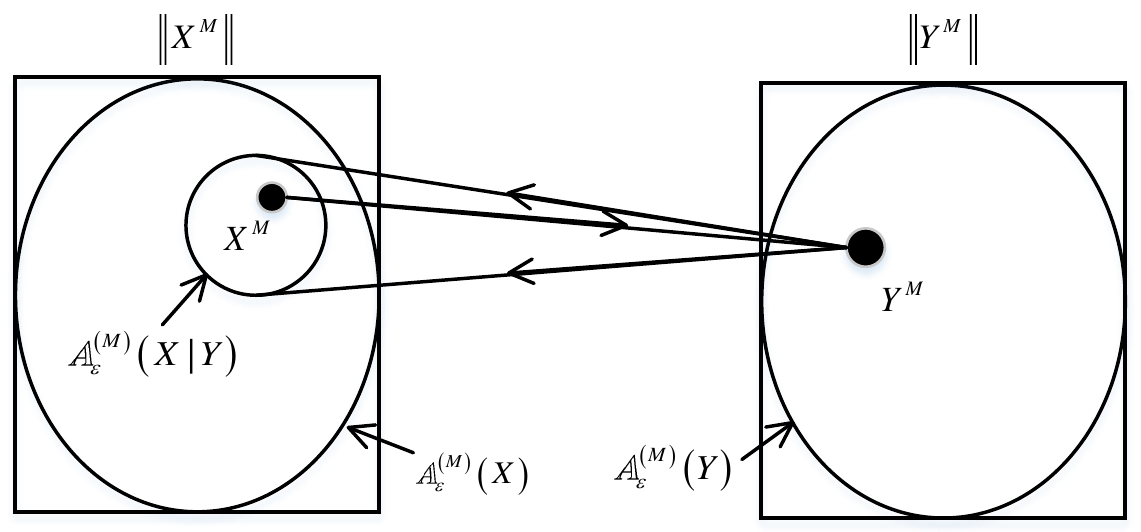}
	\begin{center}
		\caption{The typical sets and the conditional typical set for hypothesis testing}
	\end{center}
	\label{f1}
\end{figure*}

\begin{definition}[Probability of test failure]\label{2.10}
	If $ {\hat X^M} \in 	\mathbb{A}_\varepsilon ^{(M)}\left( { X|Y} \right) $, the test is said to be successful. Conversely, if $ {\hat X^M} \notin	\mathbb{A}_\varepsilon ^{(M)}\left( { X|Y} \right) $,  the test is said to be failed. The probability of test failure is defined as 
	\begin{equation}
		P_{f}^{\left( M\right) }=Pr\left\lbrace {\hat X^M} \notin	\mathbb{A}_\varepsilon ^{(M)}\left( { X|Y} \right)  \right\rbrace . 
	\end{equation}
\end{definition} 

It is noted that uncertainty remains after a successful test is made.  Define a binary random variable $ E $, where $ E=0 $ and $ E=1 $ represent the event of the successful and the failed tests, respectively. That is,

\begin{equation}
	E = \left\{ \begin{array}{l}
		0,{{\hat X}^M} \in  \mathbb{A}_\varepsilon ^{\left( M \right)}(X|Y),\\
		1,{{\hat X}^M} \notin \mathbb{A}_\varepsilon ^{\left( M \right)}(X|Y).
	\end{array} \right.
\end{equation} Hence, the probability of test failure is $P_f^{(M)} = \Pr \left\lbrace E = 1\right\rbrace $.

\begin{definition}[Empirical Entropy]\label{2.11}
	If the test is successful, i.e., $ E=0 $, the empirical entropy $ {\hat{H}}( X|Y)  $ of hypothesis testing is defined as 
	\begin{equation}
		{\hat{H}}( X|Y) = \frac{1}{M} {H}( X^{M}|Y^{M}, E=0 ).
	\end{equation} 
\end{definition} Note that the empirical entropy defined here is corresponding to specific test methods. The empirical entropy satisfies

\begin{equation}
	{\hat{H}}( X|Y) \le \frac{1}{M}\log  {\left\| {\mathbb{A}_\varepsilon ^{(M)}\left( {X|Y} \right)} \right\|} .
\end{equation}  The empirical entropy is regarded as a negative indicator of hypothesis testing. Correspondingly, below we give a definition of a positive indicator.

\begin{definition}[Accuracy]\label{2.12}
	If the test is successful, i.e., $ E=0 $, the degree of accuracy (accuracy for short) $ {\hat{I}}( X;Y) $  of hypothesis testing is defined as 
	\begin{equation}
		{\hat{I}}( X;Y)  =H\left( X \right)  - \hat{H}\left( X |Y\right) .
	\end{equation} 
\end{definition}  The accuracy $ {\hat{I}}( X;Y) $ is also called as the empirical TI. The definition of $ {\hat{I}}( X;Y) $ is rewritten as 
\begin{equation}
	\begin{array}{c}
		\begin{aligned}
			{\hat{I}}( X;Y) &\approx \frac{1}{M}{\log _2}\left\| {{2^{MH\left( X \right)}}} \right\| - \frac{1}{M}{\log _2}\left\| \mathbb{T} \right\|\\
			&= \frac{1}{M}{\log _2}\left( {\frac{{{2^{MH\left( X \right)}}}}{{\left\| \mathbb{T} \right\|}}} \right),  \label{29}
		\end{aligned}
	\end{array}
\end{equation}  where $ \mathbb{T} $ is a subset consist of  $ {\hat X^M} $. According to (\ref{29}), we will understand more intuitively the physical meaning of  accuracy $ {\hat{I}}( X;Y) $, which is of the following explanations:

(1) the only typical sequences are considered as the transmitted sequences according to the number $ {2^{MH\left( X \right)}} $ of the transmitted typical sequences in Lemma \ref{2.5}. (3). In the case of $ H(X) =1$, the number of typical sequences is ${2^M} $, i.e., all of the transmitted typical sequences are contained;

(2) the cardinality of the subset $ \mathbb{T} $  reflects the accuracy of hypothesis testing. The smaller the cardinality, the higher the  accuracy. In the case of ${\left\| \mathbb{T} \right\|}=1 $, the test?s result is the transmitted sequence,  resulting in the  highest  accuracy $ {\hat{I}}( X;Y) = H\left( X \right) $;

(3) the subset $ \mathbb{T} $   divides the typical set $ \mathbb{A}_\varepsilon ^{(M)} $  into   $  {\frac{{{2^{MH\left( X \right)}}}}{{\left\| \mathbb{T} \right\|}}} $  subsets of equal cardinality, and the sequences of length  $ {\log _2}\left( {\frac{{{2^{MH\left( X \right)}}}}{{\left\| \mathbb{T} \right\|}}} \right) $ bits are required to mark these subsets,  which are converted to $ {\hat{I}}( X;Y)  $ bits for each test. The length of the sequences is the  accuracy of the hypothesis test.

\begin{definition}[Achievability]\label{2.13}
	The accuracy is said to be achievable if there exists a test method such that the probability of test failure $ P_{f} $ tends to 0 as $ M \to \infty $.
\end{definition}

\begin{theorem}[Test Theorem]\label{2.14}
	If the TI of the hypothesis testing system $ \left( \mathbb{X}^{M}, \pi(x^{M}), p(y^{M}|x^{M}), \mathbb{Y}^{M} \right)  $ , then all the accuracy $ {\hat{I}}( X;Y) $ less than $ I(Y;X) $ are achievable. Specifically, if $ M $ is large enough, for any  $\varepsilon  > 0$, there is a test method whose accuracy $ {\hat{I}}( X;Y) $ satisfies	
	\begin{equation}
		{\hat{I}}( X;Y)  >	I(X;Y) - 2\varepsilon  
	\end{equation} and the probability of test failure is $ P_f^{(M)} \to 0 $. Conversely, there is no test method whose accuracy $ {\hat{I}}( X;Y) $ is greater than the TI given $ P_f^{(M)} \to 0 $.	
\end{theorem}
\noindent\textit{Proof: See Appendix \ref{appA} ( achievability ) and Appendix \ref{appB} ( converse to the achievability ).}

The idea of the test theorem is to determine in which typical set $\mathbb{A}_\varepsilon ^{(M)}\left( {X|Y} \right) $  the decision sequence $ \hat{X}^{M} $ is contained. For each observed sequence $ Y^{M} $, the cardinal number of the typical set  $\mathbb{A}_\varepsilon ^{(M)}\left( {X|Y} \right) $  is $ \approx {2^{MH\left( {X|Y} \right)}} $. 
Then the \textit{a posteriori} entropy of every symbol is $\approx {{\log {2^{MH\left( {X|Y} \right)}}} \mathord{\left/
		{\vphantom {{\log {2^{Mh\left( {X|Y} \right)}}} M}} \right.
		\kern-\nulldelimiterspace} M} = H\left( {X|Y} \right) $. The cardinal number of  typical set $\mathbb{A}_\varepsilon ^{(M)}\left( {X} \right) $ is  $ \approx  {2^{MH\left( X \right)}} $. Then the \textit{a priori} entropy of every symbol is $\approx {{\log {2^{MH\left( {X} \right)}}} \mathord{\left/
		{\vphantom {{\log {2^{MH\left( {X} \right)}}} }} \right.
		\kern-\nulldelimiterspace} M} = H\left( {X} \right) $.  As a result, the TI of each symbol obtained from $ Y^{M} $ is $ H\left( X\right) - H\left( {X|Y} \right) = I\left( {X;Y} \right) $. In other words, the input typical set has to be divided into subsets of number  $ {2^{M\left[ {H\left( X \right) - H\left( {X|Y} \right)} \right]}} = {2^{MI\left( {X;Y} \right)}} $ and each subset is marked with $ MI\left( {X;Y} \right) $ bit. Therefore, the TI of  $ MI\left( {X;Y} \right) $ bit is used to determine a typical set, which is converted to the TI of  $ I\left( {X;Y} \right) $ per test.

\section{Simulation}
In this section, numerical simulations are based on the example of a coin toss in the introduction for $ \theta $ is $ 0.4 $. 

Figure 6 shows the performance of the SAP test. It can be seen that the TI and the accuracy $ \hat{I}(X;Y)$  grow larger as the maximum $ N $ of the number of coins tosses $ n $ increases. This reason is that the more the number of elements in a set, the more the amount of information it can carry. Furthermore, as can be seen, there is convergence in the accuracy $ \hat{I}(X;Y)$ of the SAP test as the extension $ M $ increases. When the extension $ M $ is 10, the accuracy $ \hat{I}(X;Y)$ almost coincides with the TI. In summary, the SAP test approaches the theoretical performance as long as sufficient extension  $ M $ is provided. 
\begin{figure*}
	\centering
	\includegraphics[width=3.5in]{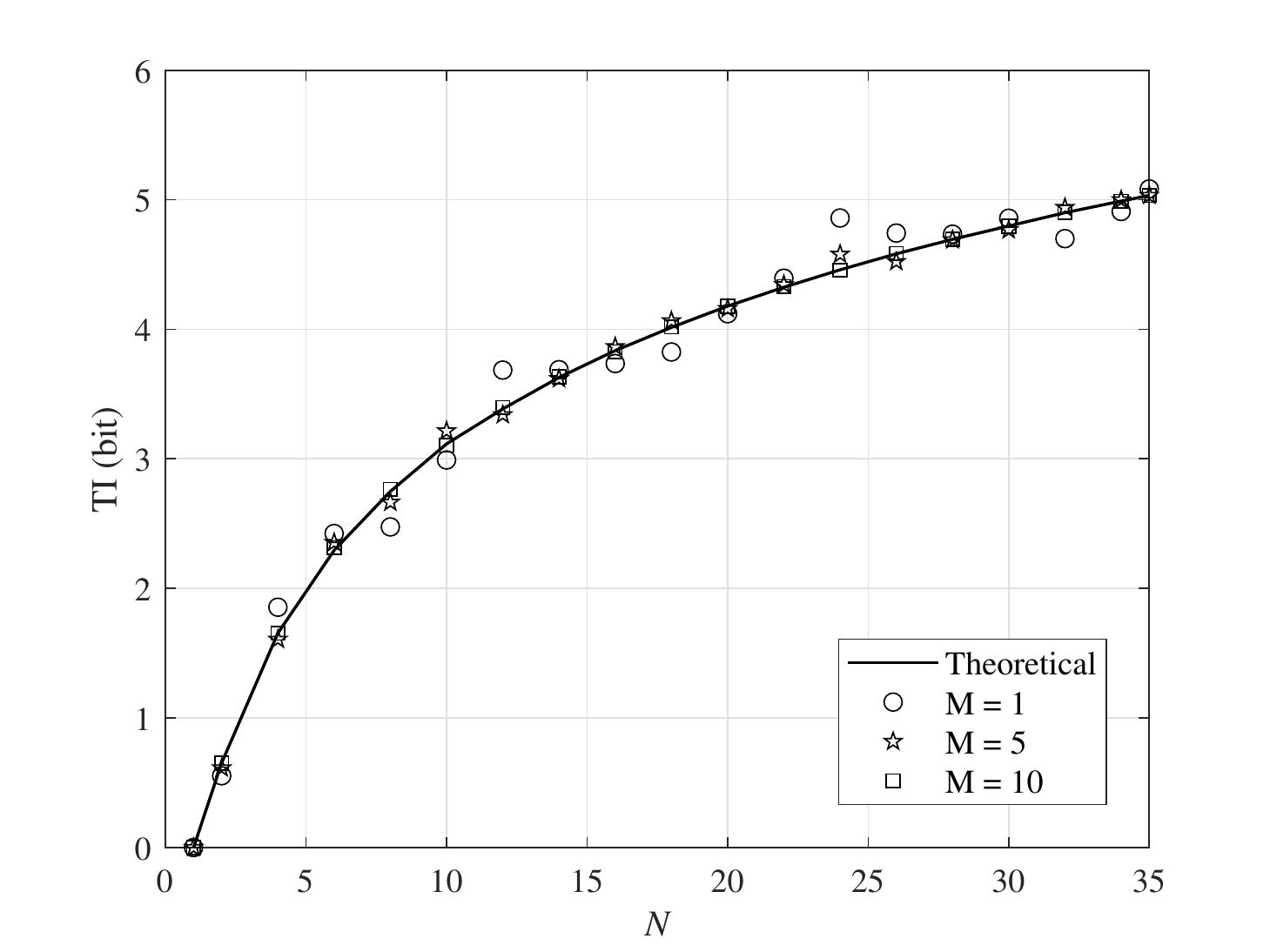}
	\caption{The convergence of the accuracy $ \hat{I}(X;Y)$ of the SAP test with the extension $ M $ increases.}
	\label{penG}
\end{figure*}

Figure 7 illustrates the comparison between the TI and the accuracy of the SAP, MAP, MeAP, and EAP tests.  It can be seen that the curves of the accuracy  of these four tests do lie below the TI curve. Moreover, the SAP test outperforms the MAP, MeAP, and EAP tests, and combined with Fig. 6, implies that the SAP test is asymptotically optimal.
\begin{figure*}
	\centering
	\includegraphics[width=3.5in]{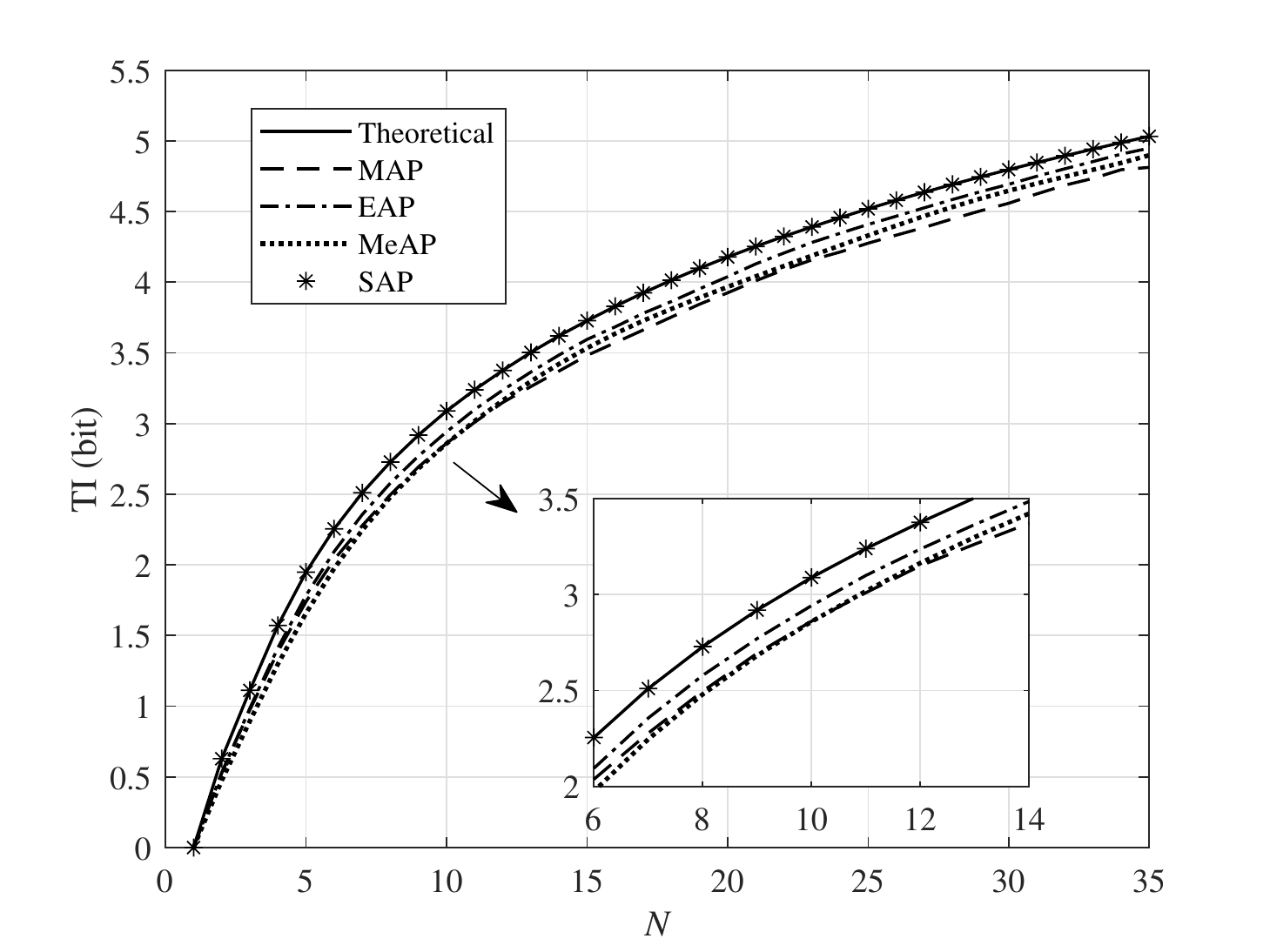}
	\caption{The accuracy $ \hat{I}(X;Y)$ of the SAP, MAP, EAP, and MeAP tests for 10 extensions are compared to the TI.}
	\label{penG}
\end{figure*}

\section{Conclusion}
The problem of optimal hypothesis testing is investigated by information theory. It is established that an information-theoretic framework of hypothesis testing consists of TI,  accuracy, SAP test, probability of test failure, and test theorem. In addition to being a significant branch of statistics, hypothesis testing is also widely applied to various fields including signal processing, psychology, and ecology. However, a critical flaw in the existing body of hypothesis testing theory is the lack of a unified evaluation indicator independent of specific test methods, leading to no optimal hypothesis testing among various test methods. The theoretical framework developed in this paper can be used to reconstruct a novel theory of hypothesis testing to promote the development of statistics and its application in related fields.

\backmatter

\section*{Declarations}

\begin{itemize}
\item Funding: this work was supported by the National Natural Science Foundation of China under Grants 62271254.
\item Competing interests: The authors have no relevant financial or non-financial interests to disclose.
\item Ethics approval: not applicable.
\item Consent to participate: informed consent was obtained from all individual participants included in the study.
\item Consent for publication: the authors agree to publish in Statistics Paper.
\item Availability of data and materials: not applicable.
\item Code availability: not applicable. 
\item Authors' contributions: Dazhuan Xu and Nan Wang performed the proof of theorem, validation, data analysis and writing.
\end{itemize}

\begin{appendices}

	\section{Proof of the target detection theorem}\label{appA}
Fix $\pi \left( x \right) $. Generate the extended state sequence  $ x^{M} $ according to the distribution,
\begin{equation}
	\pi \left( {{x^M}} \right) = \prod\limits_{m= 1}^M {\pi \left( {{x_m}} \right)}. 
\end{equation} 

The $ M $ extended conditional PDF is
\begin{equation}
	P\left( {{{y}^M}\left| {{x^M}} \right.} \right) = \prod\limits_{m = 1}^M {P\left( {{{y}_m}\left| {{x_m}} \right.} \right)} . 
\end{equation} 

Assuming that  $ P({y}^{M}|{x}^{M}) $ and the \textit{a priori} distribution  $\pi \left( x^{M} \right) $ are known, the \textit{a posteriori}  PDF is calculated by
\begin{equation}
	P(x|{y}) = \frac{{\pi (x)P({y}
			|x)}}{{\sum\limits_x {\pi (x)P({y}|x)} }}. 
\end{equation} According to Lemma \ref{2.7}, $ {\hat X^M} $ and $ {Y}^{M} $ are jointly typical sequences and the SAP test is belong to the typical set test method. By Lemma \ref{2.9}. (2), the typical set $ \mathbb{A}_\varepsilon ^M(X|Y) $ of satisfies 
\begin{equation}
	\left( {1 - \varepsilon } \right){2^{M\left( {H(X|Y) - 2\varepsilon } \right)}} < \left\| {\mathbb{A}_\varepsilon ^M( X|Y)} \right\| < {2^{M\left( {H(X|Y) + 2\varepsilon } \right)}}. 
\end{equation}

If the decision is successful, i.e.,  $ E=0 $, by lemma \ref{2.9}.(2), the empirical entropy satisfies 
\begin{equation}
	H(X|Y) - 2\varepsilon  + \frac{1}{M}\log \left( {1 - \varepsilon } \right)	< \hat{H}(X|Y) < H(X|Y) + 2\varepsilon. 
\end{equation} when $ M $ is sufficiently large, the accuracy $ {\hat{I}}( X|Y) $ satisfies
\begin{equation}
	I(X;Y) - 2\varepsilon < {\hat{I}}( X;Y)  < I(X;Y) + 2\varepsilon. 
\end{equation} The achievability of the accuracy $ {\hat{I}}( X;Y) $ is proved.

There are two events that cause the failure for the typical set test. The first is that $ X^{M} $ and $ Y^{M} $ do not form jointly typical sequences, denoted by $ {\bar A_T} $. The second is that $ \hat{X}^{M} $ and $ {Y}^{M} $  do not form jointly typical sequences, denoted by $ {\bar A_R} $. Then, the probability of test failure is 
\begin{equation}
	\begin{array}{c} 
		\begin{aligned}
			P_f^{(M)} &= \Pr \left( {{{\bar A}_T} \cup {{\bar A}_R}} \right)\\
			&\le \Pr \left( {{{\bar A}_T}} \right) + \Pr \left( {{{\bar A}_R}} \right).
		\end{aligned}
		
	\end{array}
\end{equation} According to the lemma \ref{2.9}.(3), we have
\begin{equation}
	P_f^{(M)} \le 2\varepsilon, 
\end{equation} and $P_f^{(M)} $ converges to zero as $ M $ increases.

\section{Proof of the converse theorem to the achievability}\label{appB}
\subsection{Extended Fano's Inequality}

Firstly, we provide a Lemma of extending Fano's inequality to the hypothesis testing in order to prove the converse theorem to the achievability. Focus on the conditional entropy $ H\left( {{X^M}, E|{{Y}^M}} \right) $. According to the chain rule for entropy, we have
\begin{equation}
	H\left( {{X^M},E|{Y}^M} \right) = H\left( {E|{{Y}^M}} \right) + H\left( {{X^M}|{{Y}^M},E} \right). \label{39}
\end{equation} It is obvious that $ H\left( {E|{{Y}^M}} \right) < 1 $. The remaining term $ H\left( {{X^M}|{{Y}^M},E} \right)$ can be expressed as 
\begin{equation}
	H\left( {{X^M}|{Y}^M},E \right) = \left( {1 - P_f^{(M)}} \right)H\left( {{X^M}|{{Y}^M},E = 0} \right) + P_f^{(M)}H\left( {{X^M}|{Y}^M},E = 1 \right),\label{40}
\end{equation} where $ H\left( {{X^M}|{{Y}^M},E = 0} \right)  $ denotes the uncertainty when the test is successful.

According to the property of typical set, we have
\begin{equation}
	\begin{array}{c}
		\begin{aligned}
			H\left( {{X^M}|{{Y}^M},E = 0} \right) 
			&\mathop  \le \limits^{\left( a \right)} \log \left\| {\mathbb{A}_\varepsilon ^{(M)}\left( {X|{Y}} \right)} \right\|\\
			&\mathop  \le \limits^{\left( b \right)} \log {2^{M\left[ {H\left( {X|{Y}} \right) + 2\varepsilon } \right]}}\\
			&= M\left[ {H\left( {X|{Y}} \right) + 2\varepsilon } \right],
		\end{aligned}
	\end{array}
\end{equation} where $ \left( a \right) $ is gained by the maximum discrete entropy theorem and $ \left( b \right) $ is resorted to lemma \ref{2.8}.(2). Similarly,
\begin{equation}
	\begin{array}{c}
		\begin{aligned}
			H\left( {{X^M}|{{Y}^M},E = 1} \right)&\le \log \left( {\left\| {\mathbb{A}_\varepsilon ^M(X)} \right\| - \left\| {\mathbb{A}_\varepsilon ^M(X|{Y})} \right\|} \right)\\
			&\le \log \left( {{2^{M\left[ {H\left( X \right) + \varepsilon } \right]}} - {2^{M\left[ {H\left( {X|{Y}} \right) - 2\varepsilon } \right]}}} \right)\\
			&\le \log {2^{M\left[ {H\left( X \right) + \varepsilon } \right]}}\\
			&= M\left[ {H\left( X\right) + \varepsilon } \right]. \label{42}
		\end{aligned}
	\end{array}
\end{equation} Therefore, we have the following lemma.

\begin{lemma}[Extended Fano's Inequality]\label{B.1}
	\begin{equation}
		H\left( {{X^M},E|{{Y}^M}} \right) \le 1 + \left( {1 - P_f^{(M)}} \right)H\left( {{X^M}|{{Y}^M},E = 0} \right) + P_f^{(M)}M\left[ {H\left( X \right) + \varepsilon } \right].
	\end{equation}
\end{lemma}
\begin{proof}
	According to (\ref{39}) and (\ref{40}), we have
	\begin{equation}
		H\left( {{X^M},E|{{Y}^M}} \right)= 1 + \left( {1 - P_f^{(M)}} \right)H\left( {{X^M}|{{Y}^M},E = 0} \right) + P_f^{(M)}H\left( {{X^M}|{Y}^M},E = 1 \right).
	\end{equation}
	By (\ref{42}), we have
	\begin{equation}
		H\left( {{X^M},E|{Y^M}} \right) \le 1 + \left( {1 - P_f^{(M)}} \right)H\left( {{X^M}|{Y^M},E = 0} \right) + P_f^{(M)}M\left[ {H\left( X \right) + \varepsilon } \right].
	\end{equation}
\end{proof}	

\subsection{Converse to The Test Theorem}
Fano's inequality has been extended to hypothesis testing. Then, the proof of the converse to the test theorem is provided based on Lemma \ref{B.1}. According to the properties of entropy and mutual information, we have
\begin{equation}
	H\left( {{X^M}} \right) = H\left( {{X^M}|{{Y}^M}} \right) + I\left( {{X^M};{{Y}^M}} \right), \label{46}
\end{equation} where $ H\left( {{X^M}} \right) = MH\left( X \right) $. By the property of the extension\cite{r3}, we have
\begin{equation}
	I\left( {{X^M};{{Y}^M}} \right) \le MI\left( {X;Y} \right). \label{47}
\end{equation} In light of (\ref{46}), (\ref{47}), and lemma \ref{B.1}, we have
\begin{equation}
	H\left( X \right) \le \frac{1}{M} + \frac{1}{M} \left( {1 - P_f^{(M)}} \right)H\left( {{X^M}|{Y^M},E = 0} \right) +\frac{1}{M} P_f^{(M)}M\left[ {H\left( X \right) + \varepsilon } \right] + I\left( {X;Y} \right).
\end{equation} According to the definition of the empirical entropy, we have
\begin{equation}
	H\left( X \right) - \frac{1}{M}H\left( {{X^M}|{Y^M},E = 0} \right) \le \frac{1}{M} - \frac{1}{M} P_f^{(M)}H\left( {{X^M}|{Y^M},E = 0} \right) + \frac{1}{M} P_f^{(M)}\left[ {H\left( X \right) + \varepsilon } \right] + I\left( {X;Y} \right). \label{49}
\end{equation} According to the definition of the accuracy $ {\hat{I}}( X;Y) $, (\ref{49}) is rewritten as
\begin{equation}
	{\hat{I}}( X;Y) \le \frac{1}{M} -  \frac{1}{M} P_f^{(M)}H\left( {{X^M}|{Y^M},E = 0} \right) + \frac{1}{M} P_f^{(M)}\left[ {H\left( X \right) + \varepsilon } \right] + I\left( {X;Y} \right). \label{50}
\end{equation}  If $ M \to \infty ,P_f^{(M)} \to 0 $, (\ref{50}) is expressed as
\begin{equation}
	{\hat{I}}( X;Y) < I(X;Y).
\end{equation} 




\end{appendices}


\bibliography{sn-bibliography}

\end{document}